\pdfoutput=1
\documentclass[11pt, a4paper, twoside,leqno]{amsart}
\usepackage[centering, totalwidth = 380pt, totalheight = 600pt]{geometry}
\usepackage{amssymb, amsmath, amsthm}
\usepackage{microtype, stmaryrd, url}
\usepackage[shortlabels]{enumitem}
\usepackage[latin1]{inputenc}
\usepackage[T1]{fontenc}
\usepackage{color}
\definecolor{darkgreen}{rgb}{0,0.45,0}
\usepackage[colorlinks,citecolor=darkgreen,linkcolor=darkgreen]{hyperref}
\usepackage[british]{babel}

\DeclareMathAlphabet{\mathbf}{OT1}{cmr}{b}{n}

\usepackage{tikz-cd}
\usepackage{comment}

\usepackage[arrow, matrix, tips, curve, rotate]{xy}
\SelectTips{cm}{10}

\makeatletter
\def\matrixobject@{%
  \edef \next@{={\DirectionfromtheDirection@ }}%
  \expandafter \toks@ \next@ \plainxy@
  \let\xy@@ix@=\xyq@@toksix@
  \xyFN@ \OBJECT@}
\let\xy@entry@@norm=\entry@@norm
\def\entry@@norm@patched{%
  \let\object@=\matrixobject@
  \xy@entry@@norm }
\AtBeginDocument{\let\entry@@norm\entry@@norm@patched}
\makeatother

\newcommand{\twocong}[2][0.5]{\ar@{}[#2] \save ?(#1)*{\cong}\restore}
\newcommand{\twoeq}[2][0.5]{\ar@{}[#2] \save ?(#1)*{=}\restore}
\newcommand{\rtwocell}[3][0.5]{\ar@{}[#2] \ar@{=>}?(#1)+/l 0.2cm/;?(#1)+/r 0.2cm/^{#3}}
\newcommand{\ltwocell}[3][0.5]{\ar@{}[#2] \ar@{=>}?(#1)+/r 0.2cm/;?(#1)+/l 0.2cm/^{#3}}
\newcommand{\ltwocello}[3][0.5]{\ar@{}[#2] \ar@{=>}?(#1)+/r 0.2cm/;?(#1)+/l 0.2cm/_{#3}}
\newcommand{\dtwocell}[3][0.5]{\ar@{}[#2] \ar@{=>}?(#1)+/u  0.2cm/;?(#1)+/d 0.2cm/^{#3}}
\newcommand{\dltwocell}[3][0.5]{\ar@{}[#2] \ar@{=>}?(#1)+/ur  0.2cm/;?(#1)+/dl 0.2cm/^{#3}}
\newcommand{\urtwocell}[3][0.5]{\ar@{}[#2] \ar@{=>}?(#1)+/dl  0.2cm/;?(#1)+/ur 0.2cm/^{#3}}
\newcommand{\drtwocell}[3][0.5]{\ar@{}[#2] \ar@{=>}?(#1)+/ul  0.2cm/;?(#1)+/dr 0.2cm/^{#3}}
\newcommand{\dthreecell}[3][0.5]{\ar@{}[#2] \ar@3{->}?(#1)+/u  0.2cm/;?(#1)+/d 0.2cm/^{#3}}
\newcommand{\utwocell}[3][0.5]{\ar@{}[#2] \ar@{=>}?(#1)+/d 0.2cm/;?(#1)+/u 0.2cm/_{#3}}
\newcommand{\dtwocelltarg}[3][0.5]{\ar@{}#2 \ar@{=>}?(#1)+/u  0.2cm/;?(#1)+/d 0.2cm/^{#3}}
\newcommand{\utwocelltarg}[3][0.5]{\ar@{}#2 \ar@{=>}?(#1)+/d  0.2cm/;?(#1)+/u 0.2cm/_{#3}}

\newcommand{\pullbackcorner}[1][dr]{\save*!/#1-1.2pc/#1:(-1,1)@^{|-}\restore}

\newdir{(}{{}*!<0em,-.14em>-\cir<.14em>{l^r}}
\newdir{ (}{{}*!/-5pt/\dir{(}}
\newdir{ >}{{}*!/-5pt/\dir{>}}
\newcommand{\sh}[2]{**{!/#1 -#2/}}

\newcommand{\cat}[1]{\mathsf{#1}}

\newcommand{\id}{\mathrm{id}}
\newcommand{\thg}{{\mathord{\text{--}}}}

\newcommand{\spn}[1]{{\langle{#1}\rangle}}

\newcommand{\cd}[2][]{\vcenter{\hbox{\xymatrix#1{#2}}}}

\renewcommand{\phi}{\varphi}
\newcommand{\A}{{\mathcal A}}
\newcommand{\B}{{\mathcal B}}
\newcommand{\C}{{\mathsf C}}

\newcommand{\F}{{\mathcal F}}

\newcommand{\K}{{\mathsf K}}
\renewcommand{\L}{{\mathcal L}}
\newcommand{\M}{{\mathsf M}} 

\newcommand{\R}{{\mathcal R}}
\let\sec=\S
\renewcommand{\S}{{\mathcal S}}

\newcommand{\W}{{\mathcal W}}
\newcommand{\X}{{\mathcal X}}

\renewcommand{\AA}{\mathbb{A}}
\newcommand{\LL}{\mathbb{L}}
\newcommand{\RR}{\mathbb{R}}
\newcommand{\UU}{\mathbb{U}}
\newcommand{\ZZ}{\mathbb{Z}}

\newcommand{\WE}{\mathcal{W}}
\newcommand{\Fib}{\mathcal{F}}
\newcommand{\Cof}{\mathcal{C}}
\newcommand{\rlp}[1]{{#1}^\boxslash}
\newcommand{\llp}[1]{{}^\boxslash\!{#1}}

\DeclareMathAlphabet{\mathbbe}{U}{bbold}{m}{n}

\renewcommand{\2}{\mathbbe{2}}
\newcommand{\3}{\mathbbe{3}}

\newcommand{\coalg}[1][\LL]{\mathsf{Coalg}_{#1}}
\newcommand{\alg}[1][\RR]{\mathsf{Alg}^{#1}}

\newcommand{\ecoalg}[1][(L, \vec \epsilon)]{\mathsf{Coalg}_{#1}}
\newcommand{\ealg}[1][(R, \vec \eta)]{\mathsf{Alg}^{#1}}

\newcommand{\CCoalg}[1][\LL]{\mathbb{C}\mathsf{oalg}_{#1}}
\newcommand{\AAlg}[1][\RR]{\mathbb{A}\mathsf{lg}^{#1}}

\newcommand{\Sq}[1]{\mathbb{S}\mathsf{q}(#1)}

\newcommand{\xtor}[1]{\cdl[@1]{{} \ar[r]|-{\object@{|}}^{#1} & {}}}

\makeatletter

\def\hookleftarrowfill@{\arrowfill@\leftarrow\relbar{\relbar\joinrel\rhook}}
\def\twoheadleftarrowfill@{\arrowfill@\twoheadleftarrow\relbar\relbar}
\def\leftbararrowfill@{\arrowdoublefill@{\leftarrow\mkern-5mu}\relbar\mapstochar\relbar\relbar}
\def\Leftbararrowfill@{\arrowdoublefill@{\Leftarrow\mkern-2mu}\Relbar\Mapstochar\Relbar\Relbar}
\def\leftringarrowfill@{\arrowdoublefill@{\leftarrow\mkern-3mu}\relbar{\mkern-3mu\circ\mkern-2mu}\relbar\relbar}
\def\lefttriarrowfill@{\arrowfill@{\mathrel\triangleleft\mkern0.5mu\joinrel\relbar}\relbar\relbar}
\def\Lefttriarrowfill@{\arrowfill@{\mathrel\triangleleft\mkern1mu\joinrel\Relbar}\Relbar\Relbar}

\def\hookrightarrowfill@{\arrowfill@{\lhook\joinrel\relbar}\relbar\rightarrow}
\def\twoheadrightarrowfill@{\arrowfill@\relbar\relbar\twoheadrightarrow}
\def\rightbararrowfill@{\arrowdoublefill@{\relbar\mkern-0.5mu}\relbar\mapstochar\relbar\rightarrow}
\def\Rightbararrowfill@{\arrowdoublefill@{\Relbar\mkern-2mu}\Relbar\Mapstochar\Relbar\Rightarrow}
\def\rightringarrowfill@{\arrowdoublefill@\relbar\relbar{\mkern-2mu\circ\mkern-3mu}\relbar{\mkern-3mu\rightarrow}}
\def\righttriarrowfill@{\arrowfill@\relbar\relbar{\relbar\joinrel\mkern0.5mu\mathrel\triangleright}}
\def\Righttriarrowfill@{\arrowfill@\Relbar\Relbar{\Relbar\joinrel\mkern1mu\mathrel\triangleright}}

\def\leftrightarrowfill@{\arrowfill@\leftarrow\relbar\rightarrow}
\def\mapstofill@{\arrowfill@{\mapstochar\relbar}\relbar\rightarrow}

\renewcommand*\xleftarrow[2][]{\ext@arrow 20{20}0\leftarrowfill@{#1}{#2}}
\providecommand*\xLeftarrow[2][]{\ext@arrow 60{22}0{\Leftarrowfill@}{#1}{#2}}
\providecommand*\xhookleftarrow[2][]{\ext@arrow 10{20}0\hookleftarrowfill@{#1}{#2}}
\providecommand*\xtwoheadleftarrow[2][]{\ext@arrow 60{20}0\twoheadleftarrowfill@{#1}{#2}}
\providecommand*\xleftbararrow[2][]{\ext@arrow 10{22}0\leftbararrowfill@{#1}{#2}}
\providecommand*\xLeftbararrow[2][]{\ext@arrow 50{24}0\Leftbararrowfill@{#1}{#2}}
\providecommand*\xleftringarrow[2][]{\ext@arrow 10{26}0\leftringarrowfill@{#1}{#2}}
\providecommand*\xlefttriarrow[2][]{\ext@arrow 80{24}0\lefttriarrowfill@{#1}{#2}}
\providecommand*\xLefttriarrow[2][]{\ext@arrow 80{24}0\Lefttriarrowfill@{#1}{#2}}

\renewcommand*\xrightarrow[2][]{\ext@arrow 01{20}0\rightarrowfill@{#1}{#2}}
\providecommand*\xRightarrow[2][]{\ext@arrow 04{22}0{\Rightarrowfill@}{#1}{#2}}
\providecommand*\xhookrightarrow[2][]{\ext@arrow 00{20}0\hookrightarrowfill@{#1}{#2}}
\providecommand*\xtwoheadrightarrow[2][]{\ext@arrow 03{20}0\twoheadrightarrowfill@{#1}{#2}}
\providecommand*\xrightbararrow[2][]{\ext@arrow 01{22}0\rightbararrowfill@{#1}{#2}}
\providecommand*\xRightbararrow[2][]{\ext@arrow 04{24}0\Rightbararrowfill@{#1}{#2}}
\providecommand*\xrightringarrow[2][]{\ext@arrow 01{26}0\rightringarrowfill@{#1}{#2}}
\providecommand*\xrighttriarrow[2][]{\ext@arrow 07{24}0\righttriarrowfill@{#1}{#2}}
\providecommand*\xRighttriarrow[2][]{\ext@arrow 07{24}0\Righttriarrowfill@{#1}{#2}}

\providecommand*\xmapsto[2][]{\ext@arrow 01{20}0\mapstofill@{#1}{#2}}
\providecommand*\xleftrightarrow[2][]{\ext@arrow 10{22}0\leftrightarrowfill@{#1}{#2}}
\providecommand*\xLeftrightarrow[2][]{\ext@arrow 10{27}0{\Leftrightarrowfill@}{#1}{#2}}

\makeatother

\newtheorem{Thm}{Theorem}[section]
\newtheorem{Prop}[Thm]{Proposition}
\newtheorem{Cor}[Thm]{Corollary}
\newtheorem{Lemma}[Thm]{Lemma}

\theoremstyle{definition}
\newtheorem{Defn}[Thm]{Definition}

\newtheorem{Ex}[Thm]{Example}

\newtheorem{Rk}[Thm]{Remark}
\newtheorem{Cl}[Thm]{Claim}

\makeatletter
\let\c@equation\c@Thm
\makeatother
\numberwithin{equation}{section}

\newcommand{\Mod}{\cat{Mod}}

\newcommand{\Set}{\cat{Set}}

\newcommand{\ul}{\mathrm{U}_\L}
\newcommand{\ur}{\mathrm{U}_\R}
\newcommand{\ulb}{\mathrm{U}_{\smash{\ll{\L}}}}
\newcommand{\urb}{\mathrm{U}_{\smash{\rl{\R}}}}
\makeatletter
\DeclareRobustCommand{\cev}[1]{%
  \mathpalette\do@cev{#1}%
}
\newcommand{\do@cev}[2]{%
  \fix@cev{#1}{+}%
  \reflectbox{$\m@th#1\vec{\reflectbox{$\fix@cev{#1}{-}\m@th#1#2\fix@cev{#1}{+}$}}$}%
  \fix@cev{#1}{-}%
}
\newcommand{\fix@cev}[2]{%
  \ifx#1\displaystyle
    \mkern#23mu
  \else
    \ifx#1\textstyle
      \mkern#23mu
    \else
      \ifx#1\scriptstyle
        \mkern#22mu
      \else
        \mkern#22mu
      \fi
    \fi
  \fi
}
\makeatother

\renewcommand{\ll}[1]{\cev{#1}}
\newcommand{\rl}[1]{\vec{#1}}

\setlist{leftmargin=2.5\parindent,itemsep=0.25\baselineskip}

\begin{document}

\title{Lifting accessible model structures}
 \subjclass[2010]{Primary:} \date{\today}
 \author{Richard Garner}
 \address{Centre of Australian Category Theory, Macquarie University, NSW 2109, Australia}
 \email{richard.garner@mq.edu.au}

\author{Magdalena K\k{e}dziorek}
\address{Max-Planck-Institut f\"ur Mathematik, Vivatsgasse 7, 53111 Bonn, Germany}
\email{kedziorek@mpim-bonn.mpg.de}

\author{Emily Riehl} 
\address{Department of Mathematics, Johns Hopkins University, Baltimore, MD 21218, USA}
\email{eriehl@math.jhu.edu}

\begin{abstract} 
  A Quillen model structure is presented by an interacting pair of
  weak factorization systems. We prove that in the world of locally
  presentable categories, any weak factorization system with
  \emph{accessible} functorial factorizations can be lifted along
  either a left or a right adjoint. It follows that \emph{accessible
    model structures} on locally presentable categories---ones
  admitting accessible functorial factorizations, a class that
  includes all combinatorial model structures but others besides---can
  be lifted along either a left or a right adjoint if and only if an
  essential ``acyclicity'' condition holds. A similar result was
  claimed in a paper of Hess--K\k{e}dziorek--Riehl--Shipley, but the
  proof given there was incorrect. In this note, we explain this error
  and give a correction, and also provide a new statement and a
  different proof of the theorem which is more tractable for
  homotopy-theoretic applications.
\end{abstract}

\maketitle
\tableofcontents

\newcommand{\bs}{\boldsymbol}

\section{Introduction}

In abstract homotopy theory, one often works with categories endowed
with a class $\WE$ of \emph{weak equivalences} which, though not
necessarily isomorphisms themselves, satisfy closure properties
resembling those of the isomorphisms\footnote{More precisely, $\WE$
  should contain all identities and satisfy the 2-out-of-6
  property.
}. In many cases, the category $\M$ at issue is complete, cocomplete,
and endowed with further classes of maps $\Cof$ and $\Fib$, called
\emph{cofibrations} and \emph{fibrations}, for which the pairs
\begin{equation}\label{eq:compatibility}
  (\Cof \cap\WE,\Fib) \qquad \mathrm{and} \qquad (\Cof,\Fib\cap\WE)
\end{equation}
satisfy the factorization and lifting properties axiomatized by the
notion of a \emph{weak factorization system}; see Definition
\ref{def:wfs} below. One then has a \emph{Quillen model category}: a
setting rich enough to perform many of the classical constructions of
homotopy theory.

While model structures are convenient to have, they can be difficult
to construct. One of the most useful tools for building model
structures is that of ``lifting'' an existing model structure
$(\WE, \Cof, \Fib)$ on a category $\M$ along an adjoint functor in
either one of the following situations:
\begin{equation}\label{eq:10}
  \cd[@C+1em]{
    {\C} \ar@{<-}@<-4.5pt>[r]_-{L} \ar@{}[r]|-{\top} &
    {\M} \ar@{<-}@<-4.5pt>[l]_-{U} & \text{or} & 
    {\K} \ar@{<-}@<-4.5pt>[r]_-{R} \ar@{}[r]|-{\bot} &
    {\M} \ar@{<-}@<-4.5pt>[l]_-{V}\rlap{ .}
  }
\end{equation}

On the one hand, if $U \colon \C \rightarrow \M$ is a right adjoint
functor, we may attempt to define a model structure on $\C$ by taking
the classes of weak equivalences and fibrations to be $U^{-1}(\WE)$
and $U^{-1}(\Fib)$ respectively; the model category axioms then force
the definition of the cofibrations in $\C$, since they are supposed to
provide the left class of a weak factorization system with right class
$U^{-1}(\F \cap \W)$. When these classes determine a model structure
on $\C$, we call it a \emph{right-lifting} of $(\WE, \Cof, \Fib)$
along $U$. On the other hand, if $V \colon \K \rightarrow \M$ is a
left adjoint functor, we may define weak equivalences and cofibrations
in $\K$ as the classes $V^{-1}(\WE)$ and $V^{-1}(\Cof)$, and define
the fibrations in the only way allowed by the model category axioms.
When these classes determine a model structure on $\K$, we call it a
\emph{left-lifting} of $(\WE, \Cof, \Fib)$ along $V$.

It is not always the case that right or left lifting will determine a
model structure. First, there is an essential ``acyclicity condition''
which must be satisfied, which ensures that the left and right classes
of the weak factorization systems are compatible with the
cofibrations, fibrations, and weak equivalences in the sense
of~\eqref{eq:compatibility}. In the right-lifted case, the acyclicity
condition asserts that the left class of the weak factorization system
determined by $U^{-1}(\Fib)$ (i.e.~the class of maps which are
supposed to be acyclic cofibrations) is contained in the class
$U^{-1}(\WE)$ of lifted weak equivalences. This condition is
non-trivial to check, and typically requires some genuine insight into
the homotopy theory at issue.

The other precondition for existence of the lifted model structure is
existence of the lifted weak factorization systems: and while the lifting
axiom is satisfied by construction, the existence of factorizations is
not automatic. However, there are a range of results available which verify
this existence using only general properties of the categories
involved and of the model structure $(\WE, \Cof, \Fib)$---thus
reducing the question of lifting model structures to the essential
acyclicity condition.

One situation in which lifted weak factorization systems always exist
is the combinatorial setting; here, the categories involved are
\emph{locally presentable}~\cite{Gabriel1971Lokal}---an assumption
which will remain in place for the rest of the introduction---and the
model structure $(\WE, \Cof, \Fib)$ is \emph{cofibrantly
  generated}~\cite[Definition~2.1.17]{Hovey1999Model}. In this
context, it has been understood for several decades that right-lifted
factorizations can be constructed explicitly using Quillen's small
object argument. Very recently,~\cite{makkai-rosicky} showed that in
this same setting, left-lifted factorizations also exist; this
breakthrough result was put into the model-categorical context
in~\cite[Theorem 2.23]{BHKKRS}, and has since been used to construct
interesting new model categories~\cite{ching-riehl, hess-shipley}.

These results were generalized in~\cite{HKRS} to obtain left and right
liftings of factorizations for what the authors term \emph{accessible}
model structures. The simplest formulation of what this means is that
given in~\cite{rosicky-accessible}: a model structure is accessible if
its factorizations into the classes~\eqref{eq:compatibility} can be
realized by accessible functors---ones which preserve
$\lambda$-filtered colimits for some regular cardinal $\lambda$. In
particular, any cofibrantly generated model structure is accessible,
but others besides: for example, the model structures on dg-modules
considered in~\cite{Barthel2013Six-model}. The main Corollary~3.3.4
of~\cite{HKRS} asserts that, if $(\WE, \Cof, \Fib)$ is accessible,
then both left- and right-lifted factorizations always exist; this has
already found practical application to the construction of new model
structures in~\cite{hess-kedziorek,moser}.

While the main result of~\cite{HKRS} is correct, the proof given there
turns out to contains a subtle error: in some cases, it exhibits
``lifted factorizations'' which are not those of the desired left- or
right-lifted weak factorization systems, but of slightly different
ones. The purpose of this note is to fix this error. In fact, we do so
in two ways: once by correcting the argument of~\cite{HKRS}, and once
by a different argument which sidesteps the difficulties at issue.
For good measure, we also give some concrete examples in which the
previous argument does indeed break down.

We now retrace the reasoning of~\cite{HKRS} with a view to explaining
what goes wrong. First let us note that the authors of
\emph{ibid}.~express accessibility of a model structure in a different
way to~\cite{rosicky-accessible}---taking it to mean that the two weak
factorization systems of the model structure can be made into
\emph{accessible algebraic weak factorization
  systems}~\cite{bourke-garner-awfsI}. This means that, as well as
accessible functors
\begin{equation}\label{eq:6}
  X \xrightarrow{f} Y \qquad \mapsto \qquad X \xrightarrow{Lf} Ef
  \xrightarrow{Rf} Y
\end{equation}
that realize the factorizations in each case, there should also be
provided fillers:
\begin{equation}\label{eq:4}
  \cd{
    {X} \ar[r]^-{LLf} \ar[d]_{Lf} &
    {ELf} \ar[d]^{RLf} & &
    {Ef} \ar@{=}[r] \ar[d]_{LRf} &
    {Ef} \ar[d]^{Rf} \\
    {Ef} \ar@{=}[r] \ar@{-->}[ur]_-{\delta_f} &
    {Ef} & &
    {ERf} \ar[r]_-{RRf} \ar@{-->}[ur]_-{\mu_f} &
    {Ef}
  }
\end{equation}
subject to axioms which, among other things, cause these data to
endow the functors $L \colon \M^\2 \rightarrow \M^\2$ and
$R \colon \M^\2 \rightarrow \M^\2$ with the structure of a
comonad $\LL$ and a monad $\RR$ respectively; see
Definition~\ref{def:2} below. While this is apparently stronger
than~\cite{rosicky-accessible}'s notion of accessible model category,
it turns out that, starting from the less elaborate definition, one
can always derive the data required for the more elaborate one; this
is the content of Remark~3.1.8 of~\cite{HKRS}.

The motivation for adopting the more involved definition of
accessibility is that it allows application of~\cite[Proposition
13]{bourke-garner-awfsI}, which says that accessible algebraic weak
factorization systems can always be left- and right-lifted. The
intended approach is thus the following. With $(\L, \R)$ taken
successively to be $(\Cof, \Fib \cap \WE)$ and
$(\Cof \cap \WE, \Fib)$, one first algebraizes $(\L, \R)$ to
$(\LL, \RR)$; then right- or left-lifts this along
$U \colon \C \rightarrow \M$ or $V \colon \K \rightarrow \M$ to an
algebraic weak factorization system $(\LL', \RR')$; and then takes the
\emph{underlying weak factorization system} $(\L', \R')$, whose
classes comprise the retracts of maps of the form $L'f$ or $R'f$
respectively.

This is the argument of~\cite[Corollary~3.3.4]{HKRS}; for it to work,
one must be sure that the $(\L', \R')$ produced above is indeed left-
or right-lifted from $(\L, \R)$, meaning in the left case that
$\L' = V^{-1}(\L)$, and dually in the right. This is claimed to be the
case in Theorems~3.3.1 and 3.3.2 of~\cite{HKRS}, but the claim is
incorrect. The reason is subtle, and has to do with what exactly is
lifted in applying~\cite[Proposition~13]{bourke-garner-awfsI}.

Concentrating on the left case, one lifts not the $\L$-maps of the
underlying weak factorization system, but rather the
\emph{$\mathbb{L}$-maps}: the coalgebras for the comonad $\mathbb{L}$.
While every $\mathbb{L}$-map has the property of being an $\L$-map, a
given $\L$-map may not be the underlying map of any $\mathbb{L}$-map;
it may be necessary to take a retract\footnote{A good intuition is
  that, if the $\L$-maps are the ``cofibrations'', then the
  $\mathbb{L}$-maps are the ``relative cell complexes''. In many cases
  this is literally true: see~\cite{Athorne2013Coalgebraic}.}. The
upshot of this is that, if one applies the above procedure to
$(\L, \R)$, one finds that $\L'$ comprises the retract-closure
of $V^{-1}(\mathbb{L}\text{-map})$ rather than
$V^{-1}(\L) = V^{-1}(\text{retract-closure of
}\mathbb{L}\text{-map})$; and while the former is always included in
the latter, the inclusion may be strict, as shown in
Section~\ref{sec:flaw-previous-proof} below.

In this way, the above procedure may produce factorizations for an
incorrect lifting of one of the original weak factorization systems.
The author who is responsible for this error was well aware of this
issue when she proved
\cite[Theorem~3.10]{riehl-algebraic-model-structures}---indeed, an
important part of that argument explains why it does not arise in the
cofibrantly generated and right-lifting context---but had fallen out
of touch with that awareness when writing Section~3 of~\cite{HKRS}.

Note that the problem we have described would not arise if the maps in
$\M$ admitting $\mathbb{L}$-map structure were already closed under
retracts. This observation suggests a fix: we adjust the
algebraization $(\LL, \RR)$ appropriately before lifting. Indeed, in
Proposition~\ref{prop:shifted-awfs} below, we will see that any
algebraic weak factorization system $(\LL, \RR)$ may be ``shifted'' to
one $(\LL^\sharp, \RR^\sharp)$ whose underlying weak factorization
system is the same, but whose $\LL^\sharp$-maps \emph{are} closed
under retracts. Now to correct the above procedure for left-lifting,
we need only interpolate the step of replacing $(\LL, \RR)$ by
$(\LL^\sharp, \RR^\sharp)$. Of course, exactly the same issues arise
in the case of right-lifting, and exactly the same fix is possible,
this time involving a dual shifting $(\LL^\flat, \RR^\flat)$; all of
this is detailed in Section~\ref{sec:fixing-prev-proof} below.

In addition to correcting the argument that
proves~\cite[Corollary~3.3.4]{HKRS}, the remaining aspect of this
paper is a new proof of the result which proceeds directly from the
simpler definition of accessible model structure given
in~\cite{rosicky-accessible}. In particular, this argument avoids the
use of algebraic weak factorization systems entirely, since these are
beside the point for the homotopy-theoretic applications. It is with
this more streamlined proof that we now begin the paper.

\section{The new proof}

\subsection{Background and statement of results}
\label{sec:background}

Our terminology and approach will largely follow that of~\cite{HKRS};
we begin by recalling the necessary background. Given a class of maps
$\X$ in a category $\C$, we write $\llp{\X}$ and $\rlp{\X}$ for the
classes of maps with the left, respectively right, lifting property
against each map in $\X$, and given a functor
$H \colon \C' \rightarrow \C$, we write $H^{-1}\X$ for the class of
morphisms in $\C'$ which are mapped into $\X$ by $H$.

\begin{Defn}
\label{def:wfs}
A \emph{weak factorization system} $(\L,\R)$ on a category $\C$ is
given by a \emph{left} class of maps $\L$ and a \emph{right} class of
maps $\R$ such that:
\begin{enumerate}[itemsep=0em]
\item Every morphism in $\C$ can be factored as a map in 
  $\L$ followed by one in $\R$.
\item The classes $\L$ and $\R$ are mutually determined by the equations:
  \[ \L = \llp{\R}\qquad \text{and}\qquad \R = \rlp{\L} \rlap{ ;}\] in
  the presence of the first axiom, this is equally to ask that each
  $\L$-map has the left lifting property against each $\R$-map, and
  that both $\L$ and $\R$ are closed under retracts.
\end{enumerate}

\end{Defn}

We have already discussed in the introduction what we mean by a left-
or right-lifting of a model structure along a left or right adjoint
functor; this is what Definition~2.1.3 of~\cite{HKRS} called the
\emph{left-induced} or \emph{right-induced} model structure. More
generally, we can speak of the \textit{left-lifting} or
\textit{right-lifting} of a weak factorization system $(\L, \R)$ along
a left adjoint $V$ or right adjoint $U$ as in~\eqref{eq:10}; when
these exist, they are by definition the weak factorization systems on
the domain  category with respective classes
\begin{equation}\label{eq:2}
  (\ll \L, \ll \R) = (V^{-1}\L, \rlp{(V^{-1}\L)}) \qquad \text{and}
  \qquad (\rl \L, \rl \R) = (\llp\,{(U^{-1}\R)}, U^{-1}\R)\rlap{ .}
\end{equation}

A necessary condition for the existence of a left- or right-lifted
model structure is that both of its underlying weak factorization
systems~\eqref{eq:compatibility} should admit left- or right-liftings.
Conversely, if we assume such liftings, then we have:

\begin{Prop}[Proposition~2.1.4 of~\cite{HKRS}]
  \label{prop:4}
  Let $(\WE, \Cof, \Fib)$ be a model structure on $\M$, and suppose
  there are given adjunctions as in~\eqref{eq:10} for which the
  right-lifted weak factorization systems exist on $\C$ and the
  left-lifted weak factorization systems exist on $\K$. In this
  situation:
  \begin{enumerate}
  \item The right-lifted model structure exists on $\C$ if and only
    if the \textbf{right acyclicity condition} $\llp\,{(U^{-1}\Fib)} \subseteq
    U^{-1}\WE$ is satisfied.
  \item The left-lifted model structure exists on $\C$ if and only if
    the \textbf{left acyclicity condition}
    $\rlp{(V^{-1}\Cof)} \subseteq V^{-1}\WE$ is satisfied.
  \end{enumerate}
\end{Prop}
As noted in the introduction, the satisfaction of the acyclicity
condition typically depends on non-trivial homotopy-theoretic
arguments; this is discussed at some length in~\cite[\sec 2.2] {HKRS}.
In this paper, however, our sole interest will be in verifying the
existence of the lifted weak factorization systems as in~\eqref{eq:2}. 
The setting in which we do so is that of \emph{accessible weak factorization systems}.

\begin{Defn}
  \label{def:3}
  A weak factorization system $(\L, \R)$ on a category $\M$ is called
  \emph{accessible} if $\M$ is locally presentable, and there is given
  a functorial realization
    \begin{equation}\label{eq:7}
    A \xrightarrow{f} B \qquad \mapsto \qquad A \xrightarrow{Lf} Ef
    \xrightarrow{Rf} B
  \end{equation}
  for $(\L, \R)$ whose underlying functor
  $E \colon \M^\2 \rightarrow \M$ is accessible.\footnote{By the usual
    retract argument, a given functorial factorization provides
    factorizations for at most one weak factorization system: namely,
    that whose left and right classes comprise the retracts of the
    $Lf$'s and the $Rf$'s respectively. This is the sense in which we
    refer to $(L,R)$ as a \emph{functorial realization} of the weak
    factorization system $(\L,\R)$.} A model structure on $\M$ is
  \emph{accessible} if its underlying weak factorization
  systems~\eqref{eq:compatibility} are so.
\end{Defn}

The key objective of this paper is to give a correct proof of:
\begin{Thm}\label{thm:2}
  Let $(\L, \R)$ be an accessible weak factorization system on $\M$,
  and suppose there are given adjunctions~\eqref{eq:10}
  with $\C$ and $\K$ also locally presentable. In these circumstances,
  $(\L, \R)$ admits a left-lifting along $V$ and a right-lifting along
  $U$, and these are again accessible.
\end{Thm}
Using this, we re-find the main Corollary~3.3.4 of~\cite{HKRS}:
\begin{Cor}
  \label{cor:1}
  Let $(\WE, \Cof, \Fib)$ be an accessible model structure on $\M$,
  and suppose given adjunctions~\eqref{eq:10}
  with $\K$ and $\C$ also locally presentable.
  \begin{enumerate}
  \item The right-lifted model structure exists on $\C$ if and only if
    the right acyclicity condition holds.
  \item The left-lifted model structure exists on $\K$ if and only if
    the left acyclicity condition holds.
  \end{enumerate}
\end{Cor}

\subsection{Cloven $\L$- and $\R$-maps}
\label{sec:cloven-l-maps}
The first proof we give of Theorem~\ref{thm:2} will still employ ideas
derived from~\cite{bourke-garner-awfsI}, but will be given in a fully
self-contained manner with the minimum of additional machinery. The
main notion we require is:


%

\begin{Defn}
  \label{def:1}
  Let $(\L, \R)$ be an accessible weak factorization system on $\M$. A
  \emph{cloven} \emph{$\L$-map} $(f, s)\colon A \rightarrow B$
  comprises a map $f \colon A \rightarrow B$ of $\M$ together with a
  lift of $f$ against its own right factor, as to the left in:
\begin{equation}\label{eq:14}
  \cd{
    {A} \ar[r]^-{Lf} \ar[d]_{f} &
    {Ef} \ar[d]^{Rf} & &
    {C} \ar[d]_-{Lg} \ar@{=}[r] &
    {C} \ar[d]^{g}
    \\
    {B} \ar@{=}[r] \ar@{-->}[ur]_-{s} &
    {B} & &
    {Eg} \ar@{-->}[ur]^-{p} \ar[r]_-{Rg} &
    {D}\rlap{ .}
  }
\end{equation}
Dually, a \emph{cloven $\R$-map} $(g,p) \colon C \rightarrow D$ is a
map $g \colon C \rightarrow D$ together with a lift of $g$ against its
own left factor, as above right. The cloven $\L$-maps are the objects
of a category $\cat{Clov}(\L)$, wherein a morphism
$(f,s) \rightarrow (g,t)$ is a map $(h,k) \colon f \rightarrow g$ in
$\M^\2$ as below left which also renders commutative the square to the
right:
\begin{equation*}
  \cd{
    {A} \ar[r]^-{h} \ar[d]_{f} &
    {C} \ar[d]^{g} & & 
    {Ef} \ar[r]^-{E(h,k)} &
    {Eg} \\
    {B} \ar[r]_-{k} &
    {D} & &
    {B} \ar[r]_-{k} \ar[u]^{s} &
    {D} \ar[u]_{t}\rlap{ .}
  }
\end{equation*}
Dually, the cloven $\R$-maps form a category
$\cat{Clov}(\R)$. We write
$\ul \colon \cat{Clov}(\L) \rightarrow \M^\2$ and
$\ur \colon \cat{Clov}(\R) \rightarrow \M^\2$ for the functors
forgetting the liftings.
\end{Defn}

It will be useful to re-express the above definition in a different
manner. Any functorial factorization~\eqref{eq:6} yields endofunctors
$L, R \colon \M^\2 \rightarrow \M^\2$ and natural transformations
$\vec \eta \colon \id_{\M^\2} \Rightarrow R$ and
$\vec \epsilon \colon L \Rightarrow \id_{\M^\2}$ with components
\begin{equation}\label{eq:9}
  {\vec\eta}_f = (Lf,1) \colon f \rightarrow Rf \qquad \text{and} \qquad {\vec\epsilon}_f = (1,Rf)
  \colon Lf \rightarrow f\rlap{ .}
\end{equation}
In these terms, to endow $g \colon C \rightarrow D$ with cloven
$\R$-map structure is to endow it with a choice of retraction
$\vec p \colon Rg \rightarrow g$ for ${\vec\eta}_g$, or in other
words, with \emph{$(R, \vec \eta)$-algebra} structure. We may thus
identify $\ur$ with the forgetful functor
$\smash{\ealg} \rightarrow \M^\2$ from the category
of $(R, \vec \eta)$-algebras. Similarly, we may identify $\ul$ with the forgetful
functor $\smash{\ecoalg} \rightarrow \M^\2$
from the category of $(L, \vec \epsilon)$-coalgebras.

\begin{Lemma}
  \label{lem:4}
  Let $(\L, \R)$ be an accessible weak factorization system on $\M$.
  \begin{enumerate}
  \item $\ul \colon \cat{Clov}(\L) \rightarrow \M^\2$ is a left
    adjoint isofibration between locally presentable categories, and
    the objects in its image are precisely the $\L$-maps.
  \item $\ur \colon \cat{Clov}(\R) \rightarrow \M^\2$ is a right
    adjoint isofibration between locally presentable categories, and
    the objects in its image are precisely the $\R$-maps.
  \end{enumerate}
\end{Lemma}
Recall here that a functor $F \colon \A \rightarrow \B$ is an
\emph{isofibration} when, for every isomorphism
$f \colon b \rightarrow Fa$ in $\B$, there exists an isomorphism
$f' \colon a' \rightarrow a$ in $\A$ with $Ff' = f$. These are the
fibrations of the ``folk'' model structure on
$\cat{CAT}$~\cite{Joyal1991Strong}.
\begin{proof}
  It follows from the identification of $\mathsf{Clov}(\L)$ and
  $\mathsf{Clov}(\R)$ with $\smash{\ecoalg}$ and $\smash{\ealg}$ that
  $\ul$ and $\ur$ are isofibrations, that $\ul$ creates colimits and
  that $\ur$ creates limits (cf.~\cite[Theorem
  3.4.2]{Barr1985Toposes}). In particular, $\cat{Clov}(\L)$ is
  complete and $\cat{Clov}(\R)$ cocomplete, and so
  by~\cite[Theorem~2.47]{Adamek1994Locally}, both will be locally
  presentable so long as they are \emph{accessible
    categories}~\cite{Makkai1989Accessible}. We show this using
  Theorem~5.1.6 of \emph{ibid}., which states that the 2-category
  $\cat{ACC}$ of accessible categories and accessible functors is
  closed in $\cat{CAT}$ under bilimits. This implies the accessibility
  of $\cat{Clov}(\L)$ and $\cat{Clov}(\R)$, because the passage from a
  (co)pointed endofunctor to its category of (co)algebras can be
  realized using bilimits (cf.~\cite[Appendix~A]{ching-riehl}), and
  because the accessibility of $E$ implies that $L, R, \vec \eta$ and
  $\vec \epsilon$ all live in $\cat{ACC}$.

  Now $\ul$ is a cocontinous functor between locally presentable
  categories, and so by~\cite[Theorem~5.33]{Kelly1982Basic} has a
  right adjoint; while $\ur$ is a continuous and
  accessible functor between locally presentable
  categories---accessible due to its construction from bilimits
  in $\cat{ACC}$---and so by~\cite[Satz~14.6]{Gabriel1971Lokal} has a
  left adjoint.

  For the final claim, if $f$ is an $\L$-map, then it lifts against
  $Rf$, and so admits a cleavage; conversely, if $f$ is endowed with a
  cleavage, then it is an $\L$-map as a retract of the $\L$-map $Lf$.
  So the image of $\ul$ comprises precisely the $\L$-maps, and dually
  the image of $\ur$ comprises the $\R$-maps.
\end{proof}
We will require one final result relating to cloven maps. We state it
here only for $\L$-maps, and leave the  dualization to the
right case to the reader.

\begin{Lemma}
  \label{lem:3}
  Let $(f,s) \colon A \rightarrow B$ be a cloven $\L$-map and
  $g \colon B \rightarrow C$ an $\L$-map. There is a cleavage $t$ for
  $gf \colon A \rightarrow C$ such that
  $(1,g) \colon (f,s) \rightarrow (gf,t)$ in $\cat{Clov}(\L)$.
\end{Lemma}
\begin{proof}
  Take $t$ be any filler for the square
  \begin{equation*}
    \cd[@C+1em]{
      {B} \ar[r]^-{E(1,g).s} \ar[d]_{g} &
      {E(gf)} \ar[d]^{R(gf)} \\
      {C} \ar@{=}[r] \ar@{-->}[ur]_-{t}&
      {C}\rlap{ .}
    }
  \end{equation*}
  We have $R(gf) \circ t = 1_C$ and
  $tgf = E(1,g) \circ sf = E(1,g) \circ Lf = L(gf)$, so $t$ is a
  cleavage. Moreover, $(1,g) \colon (f,s) \rightarrow (gf,t)$ is a map
  in $\cat{Clov}(\L)$ by commutativity of the top triangle above.
\end{proof}

\subsection{Lifting accessible weak factorization systems}
\label{sec:lift-access-weak}

We are now ready to give our first proof of Theorem~\ref{thm:2}. 
In order to exhibit the desired factorizations into the lifted
classes, we consider the following pullback diagrams:
\begin{equation}
  \label{eq:3}
  \cd{
    {\cat{Clov}(\ll\L)} \ar[r]^-{} \ar[d]_{\ulb} \pullbackcorner &
    {\cat{Clov}(\L)} \ar[d]^{\ul} & &
    {\cat{Clov}(\rl\R)} \ar[r]^-{} \ar[d]_{\urb} \pullbackcorner &
    {\cat{Clov}(\R)} \ar[d]^{\ur} \\
    {\K^\2} \ar[r]_-{V^\2} &
    {\M^\2} & &
    {\C^\2} \ar[r]_ -{U^\2} &
    {\M^\2\rlap{ .}}
  }
\end{equation}
The notation for the categories defined by these pullbacks is slightly
abusive; the meaning cannot be the one asserted by
Definition~\ref{def:1}, since we do not yet have functorial
factorizations for $\smash{(\ll \L, \ll \R)}$ or
$\smash{(\rl \L, \rl \R)}$. Indeed, the whole point is to find such
factorizations, and we will do this with the aid of the above
pullbacks.

The abuse of notation is justified by the observation that an object
of, say, $\cat{Clov}(\smash{\ll{\L}})$ is a pair $(f,s)$ where $f$ is
a map of $\K$ and $s$ is a cleavage for $Vf$---thus, by
Lemma~\ref{lem:4}, a witness that $Vf$ is an $\L$-map and so equally,
a witness that $f$ is an $\smash{\ll \L}$-map. This proves the final
clauses in the two parts of the following result.

\begin{Lemma}
  \label{lem:new4}
  Let $(\L, \R)$ be an accessible weak factorization system on $\M$,
  and suppose given adjunctions~\eqref{eq:10} with $\C$ and $\K$ also
  locally presentable.
  \begin{enumerate}
  \item $\ulb \colon \cat{Clov}(\ll\L) \rightarrow \K^\2$ is a left
    adjoint isofibration between locally presentable categories, and
    the objects in its image are precisely the $\smash{\ll\L}$-maps.
  \item $\urb \colon \cat{Clov}(\rl\R) \rightarrow \C^\2$ is a right
    adjoint isofibration between locally presentable categories, and
    the objects in its image are precisely the $\smash{\rl\R}$-maps.
  \end{enumerate}
\end{Lemma}
\begin{proof}
  It remains to prove the first clauses. By Lemma~\ref{lem:4}, $\ul$
  is an isofibration; whence by~\cite{Joyal1993Pullbacks}, its
  pullback along $V^\2$ is also a bipullback (= homotopy pullback in
  $\cat{CAT}$). By~\cite[Theorem~3.15]{Bird1984Limits}, the
  $2$-category of locally presentable categories and left adjoint
  functors is closed under bilimits in $\cat{CAT}$, so that, in
  particular, $\smash{\ulb}$ is a left adjoint between locally
  presentable categories. Similarly,
  by~\cite[Theorem~2.18]{Bird1984Limits}, the $2$-category of locally
  presentable categories and \emph{right} adjoint functors is closed
  under bilimits in $\cat{CAT}$, and so $\smash{\urb}$ is a right
  adjoint between locally presentable categories.
\end{proof}
We now show that the adjoints asserted by this lemma provide the
desired functorial $\smash{(\ll \L, \ll \R)}$- and
$\smash{(\rl \L, \rl \R)}$-factorizations. The argument from this
point is completely dualizable, so we concentrate on the case of
left-lifting.

\begin{Prop}
  \label{prop:2}
  Under the hypotheses of Theorem~\ref{thm:2}, the counit of the
  adjunction
  $\ulb \dashv G \colon \K^\2 \rightarrow \cat{Clov}(\smash{\ll
    \L})$ at an object $f$ may be taken to be of the form:
  \begin{equation}\label{eq:1}
    \cd{
      {X} \ar@{=}[r]^-{} \ar[d]_{\ll Lf} &
      {X} \ar[d]^{f} \\
      {\ll Ef} \ar[r]_-{\ll Rf} & 
      {Y}
    }
  \end{equation}
\end{Prop}
\begin{proof}
  It suffices to prove that, for any right adjoint $G$ for $\ulb$, the
  counit maps $\ulb G f \rightarrow f$ have invertible
  domain-components; then we may transport the values of $G$ along
  these invertible maps to get a right adjoint with counit as
  in~\eqref{eq:1}.
  
  So let $(g,s) \in \cat{Clov}(\ll \L)$ be the value at $f$ of some
  right adjoint for $\ulb $, with the corresponding counit map given by
  the square in $\K$ left below.
  \begin{equation*}
    \cd[@C+0.3em]{
      {X'} \ar[r]^-{x} \ar[d]_-{g} &
      {X} \ar[d]^-{f} &
      1_{X'} \ar[d]_-{(x, x)} \ar[r]^-{(1,g)} & g \ar[d]^-{(x,y)} &
\ulb (1_{X'},L1_{FX'}) \ar[d]_-{\ulb (x, x)} \ar[r]^-{\ulb (1,g)} &
\ulb (g,s) \ar[d]^-{(x,y)}
      \\
      {Y'} \ar[r]_-{y} &
      {Y} &
      1_X \ar[r]_-{(1,f)} & f &
\ulb (1_X, L1_{FX}) \ar[r]_-{(1,f)} \ar@{-->}[ur]|-{\ulb (z,w)} & f
    }
  \end{equation*}
  This square in $\K$ yields one in $\K^\2$ as centre above, and a
  short calculation shows that we can lift its top and left sides to
  $\cat{Clov}(\ll{\L})$ as in the solid square right above. Since
  the counit $(x,y)$ is, by definition, terminal in the comma category
  $\ulb \downarrow f$, we induce a unique diagonal filler as displayed
  making both triangles commute. In particular, both $zx = 1$ and
  $xz = 1$ so $x$ is invertible as desired.
\end{proof}
The naturality of the counit means that the factorization
$f \mapsto (\ll Lf, \ll Rf)$ in~\eqref{eq:1} is functorial; and
in fact, as the notation suggests, we have:
\begin{Prop}
  \label{prop:6}
  Under the hypotheses of Theorem~\ref{thm:2}, the
  factorization~\eqref{eq:1} is a functorial
  $(\ll \L, \ll \R)$-factorization.
\end{Prop}
\begin{proof}
  The diagram~\eqref{eq:1} provides the counit at $f \colon X
  \rightarrow Y$ of an adjunction $\ulb \dashv G \colon
  \K^\2 \rightarrow \cat{Clov}(\ll\L)$. In particular, each $\ll Lf =
  \ulb Gf$ is in the image of $\ulb$ and so is an $\ll \L$-map by Lemma \ref{lem:new4}. It remains to show each $\ll Rf \in \ll \R$. 

  We can write $Gf = (\ll Lf, s)$, where $s$ is a cleavage for
  $V\ll Lf \colon VX \rightarrow V \ll Ef$. Now since
  $V \ll L \ll R f \colon V \ll E f \rightarrow V \ll E \ll R f$ is an
  $\L$-map, there is by Lemma~\ref{lem:3} a cleavage $t$ for
  $V(\ll L \ll Rf.\ll Lf)$ such that
  $(1,\ll L \ll R f) \colon (\ll Lf, s) \rightarrow (\ll L \ll Rf.\ll
  Lf, t)$ in $\cat{Clov}(\ll \L)$. This gives a square as to the left
  of:
  \begin{equation*}
    \cd{
      \ulb (\ll Lf, s) \ar[r]^-{\ulb (1,1)} \ar[d]_-{\ulb (1, \ll L \ll Rf)} &
      \ulb (\ll Lf, s) \ar[d]^-{(1, \ll Rf)} & &
      \ll Ef \ar@{=}[r] \ar[d]_-{\ll L \ll Rf} & \ll Ef
      \ar[d]^-{\ll Rf}
      \\
      \ulb (\ll L \ll R f. \ll Lf, t) \ar[r]_-{(1, \ll R \ll Rf)}
      \ar@{-->}[ur]|-{\ulb (1, \mu)}&
      f & &
      \ll E \ll Rf \ar@{-->}[ur]^-{\mu} \ar[r]_-{\ll R \ll Rf} & Y
    }
  \end{equation*}
  in $\K^\2$. Since the counit is terminal in
  $\ulb \downarrow f$, we induce a unique diagonal filler as
  displayed making both triangles commute; and on taking the codomain
  projection, we obtain the commuting diagram above right. We are now
  ready to show that $\ll Rf \in \ll \R = \rlp{\ll \L}$. So suppose
  $\ell \in \ll \L$, and we are given a square as to the left in:
  \begin{equation*}
    \cd{
      {A} \ar[r]^-{h} \ar[d]_{\ell} &
      {\ll Eg} \ar[d]^{\ll Rg} & &
      {A} \ar[r]^-{h} \ar[d]_{\ell} &
      {\ll Eg} \ar[d]^-{\ll L\ll Rf} \ar@{=}[r]^-{} & \ll Ef
      \ar[d]^-{\ll Rf}
      \\
      {B} \ar[r]_-{k} &
      {Y} & &
            {B} \ar[r]_-{k'} &
      {\ll E \ll Rf} \ar[r]_-{\ll R \ll Rf} & Y\rlap{ .}    }
  \end{equation*}
  Choose a cleavage $r$ for $V\ell$. By terminality of the counit
  $(1, \ll R \ll Rf) \colon \ll L \ll R f \rightarrow \ll R f$ in
  $\ulb \downarrow \ll Rf$, we may now factor
  $(h,k) \colon \ulb(\ell,r) \rightarrow \ll Rf$ as right above, and
  so obtain the desired filler as
  $\mu k' \colon B \rightarrow \ll Ef$.
\end{proof}
Putting the above results together, we obtain:
\begin{proof}[Proof of Theorem~\ref{thm:2}]
  In the left-lifted case, Proposition~\ref{prop:6}
  exhibits~\eqref{eq:1} as a functorial
  $(\ll \L, \ll \R)$-factorization, so that the lifted weak
  factorization system exists. To show accessibility, it suffices to
  show that $\ll E$ in~\eqref{eq:1} is an accessible functor; but this
  is so since it is the composite of accessible functors:
  \begin{equation*}
    \K^\2 \xrightarrow{G} \cat{Clov}(\ll \L) \xrightarrow{\ulb}
    \K^\2 \xrightarrow{\mathrm{cod}} \K\rlap{ .}
  \end{equation*}
  The case of right-lifting is entirely dual.
\end{proof}

\section{The previous proof}
\label{sec:previous-proof}

In the rest of the paper, we revisit the proof of Theorem~\ref{thm:2}
given in~\cite{HKRS} in order to explain where it goes wrong, and to
suggest a way of fixing it. This proof starts from a different, though
equivalent, formulation of accessibility for a model structure, given
in terms of the \emph{algebraic weak factorization systems}
of~\cite{Grandis2006Natural}, and we begin by explaining this.

\subsection{Accessible algebraic weak factorization systems}
\label{sec:algebr-weak-fact}
Lemma~\ref{lem:4} tells us that we can recapture a weak factorization
system $(\L, \R)$ from any of its functorial realizations $(L,R)$:
indeed, $\L$ and $\R$ are the classes of maps admitting
$(L, \vec \epsilon)$-coalgebra, respectively
$(R, \vec \eta)$-algebra, structure. However, not every functorial
factorization realizes a weak factorization system; the additional
structure required to ensure this was identified
in~\cite[Theorem~2.4]{Rosicky2002Lax-factorization}:

\begin{Lemma}
  \label{lem:1}
  A functorial factorization~$(L,R)$ realises a weak factorization
  system $(\L, \R)$ if and only if each $Lf$ admits
  $(L, \vec \epsilon)$-coalgebra structure and each $Rf$ admits
  $(R, \vec \eta)$-algebra structure.
\end{Lemma}
Choosing such coalgebra and algebra structures amounts to choosing
sections $\vec \delta_f \colon Lf \rightarrow LLf$ for each
$\vec \epsilon_f$, and retractions
$\vec \mu_f \colon RRf \rightarrow Rf$ for each $\vec \eta_f$; or, in
more elementary terms, to choosing fillers
$\delta_f \colon Ef \rightarrow ELf$ and
$\mu_f \colon ERf \rightarrow Rf$ for all squares as in~\eqref{eq:4}.
If this is done carefully enough, we may obtain an instance of the
following structure.


\begin{Defn}\label{def:2}
  An \emph{algebraic weak factorization system} $(\LL,\RR)$ on a
  category $\M$ comprises a comonad
  $\LL = \smash{(L,\vec{\epsilon},\vec{\delta})}$ and a monad
  $\RR = \smash{(R,\vec{\eta}, \vec{\mu})}$ on $\M^\2$ such that
  $L, R, \vec \epsilon$ and $\vec \eta$ arise from a functorial
  factorization~\eqref{eq:7} in the manner of~\eqref{eq:9}, and such
  that the canonical map $(\delta,\mu) \colon LR \Rightarrow RL$ is a distributive law.
\end{Defn}

By Lemma~\ref{lem:1}, any algebraic weak factorization system
$(\LL, \RR)$ has an underlying weak factorization $(\L, \R)$ whose
classes are the maps admitting $(L, \vec \epsilon)$-coalgebra or
$(R, \vec \eta)$-algebra structure. However, equally important in this
context are the \emph{$\LL$-maps} and \emph{$\RR$-maps}: the
coalgebras for the comonad $\LL$ and the algebras for the monad $\RR$.
The \emph{data} for $\LL$- or $\RR$-map structure is the same as that
for $(L, \vec \epsilon)$-coalgebra or $(R, \vec \eta)$-algebra
structure---a choice of filler as to the left or right
in~\eqref{eq:14}---but an additional (co)associativity axiom is
required; so not every $\L$- or $\R$-map need admit $\LL$- or
$\RR$-map structure. The general situation is that:
\begin{Lemma}
  \label{lem:2}
  If $(\LL, \RR)$ is an algebraic weak factorization system, then its
  underlying weak factorization system has classes
  $\L = \mathrm{Retr}(\exists \LL)$ and
  $\R = \mathrm{Retr}(\exists \RR)$, where we write
  $\mathrm{Retr}(\thg)$ for the operation of retract-closure, and
  write
  \begin{align*}
    \exists \LL &= \{f \in \M^\2 : \text{$f$ admits $\LL$-map
      structure}\} \\ \text{and} \qquad
   \exists \RR &= \{g \in \M^\2 : \text{$g$ admits $\RR$-map structure}\}.
  \end{align*}
\end{Lemma}
\begin{proof}
  Each $\LL$-map is a fortiori an $(L, \vec\epsilon)$-coalgebra and so
  has underlying map in $\L$; whence
  $\mathrm{Retr}(\exists \LL) \subseteq \mathrm{Retr}(\L) = \L$.
  Conversely, each $\L$-map admits by Lemma~\ref{lem:4} a coalgebra
  structure exhibiting it as a retract of $Lf$; as $Lf$ underlies the
  $\LL$-map $(Lf, \smash{\vec \delta_f})$, we thus have
  $\L \subseteq \mathrm{Retr}(\exists \LL)$. The right case is 
  dual.
\end{proof}

\begin{Rk}
  \label{rk:3}
  It was shown in~\cite{Garner2009Understanding} that, in the locally
  presentable setting, each weak factorization system $(\L, \R)$
  generated by a set of maps $J$ has an algebraic realization
  $(\LL, \RR)$, in which the $\RR$-maps are morphisms
  $f \colon X \rightarrow Y$ equipped with chosen lifts against each
  map in $J$. In this case, we have $\exists \RR = \R$, but typically
  $\exists \mathbb{L} \subsetneq \L$; in fact, $\exists \mathbb{L}$
  often comprises precisely the \emph{$J$-cell complexes} of which the
  $\L$-maps are retracts (cf.~\cite{Athorne2013Coalgebraic}). On the
  other hand,~\cite[Proposition~17]{bourke-garner-awfsI} gives an
  example of an algebraic weak factorization system on $\cat{CAT}$ for
  which $\exists \RR \subsetneq \R$.
\end{Rk}

We say that an algebraic weak factorization system $(\LL, \RR)$ is
\emph{accessible} if $\M$ is locally presentable and the functor
$E \colon \M^\2 \rightarrow \M$ underlying the functorial
factorization is accessible. In this circumstance, the underlying weak
factorization system is clearly accessible. In the other direction, we
have the following result; for the proof, see \sec 3.1 of~\cite{HKRS},
in particular Remark~3.1.8.
\begin{Prop}
  \label{prop:1}
  Every accessible weak factorization system is the underlying weak
  factorization system of an accessible algebraic weak factorization
  system.
\end{Prop}

In light of this, we can equally define an accessible model structure
on a locally presentable category $\M$ to be one whose underlying weak
factorization systems admit accessible algebraic realizations. This
is the choice made in~\cite[Definition~3.1.6]{HKRS}, in order to
exploit known results on lifting accessible algebraic weak
factorization systems; it is to these that we now turn.

\subsection{Lifting algebraic weak factorization systems}
\label{sec:lift-algebr-weak-1}
To explain left- and right-lifting of algebraic weak
factorization systems, we first need to recall the manner in which
$\LL$- and $\RR$-maps \emph{compose}. This is governed by certain
functors into the categories $\coalg$ and $\alg$ of
$\LL$- and $\RR$-maps, as in the dotted parts of:
\begin{equation*}
  \cd[@C=4pc@R=2pc]{
    \coalg \times_{\M} \coalg \ar@{-->}[r]^-{\circ}
    \ar[d]_-{\mathrm{U}_{\LL} \times_\M \mathrm{U}_{\LL}}
    & \coalg \ar[d]^{\mathrm{U}_{\LL}} \ar@<-1ex>[r]_{t\mathrm{U}_\LL}
    \ar@<1ex>[r]^{s\smash{\mathrm{U}_\LL}} & \M \ar@{-->}[l]|{i} \ar[d]^-{\id} &
    \CCoalg \ar[d]^-{\UU_\LL}\\
    \M^\3 \ar[r]^-\circ & \M^\2 \ar@<-1ex>[r]_{t} \ar@<1ex>[r]^{s} & \M
    \ar[l]|{\,i\,} &
    \Sq\M \\
    \alg \times_{\M} \alg \ar@{-->}[r]^-{\circ}
    \ar[u]^-{\mathrm{U}_{\RR} \times_\M \mathrm{U}_{\RR}}
    & \alg \ar[u]_{\mathrm{U}_{\RR}} \ar@<-1ex>[r]_{t\smash{\mathrm{U}_\RR}}
    \ar@<1ex>[r]^{s\mathrm{U}_\RR}& \M \ar@{-->}[l]|{i} \ar[u]_-{\id} & \AAlg
    \ar[u]_-{\UU_\RR}\rlap{ .}
  }
\end{equation*}
These functors exhibit the top and bottom rows as \emph{double
  categories}---i.e., internal categories in $\cat{CAT}$---over the
double category $\Sq{\M}$ of objects, morphisms, morphisms and
commutating squares in $\M$. We display this to the right above. In
more detail, objects and horizontal morphisms of these double
categories $\CCoalg$ and $\AAlg$ are just objects and arrows of $\M$;
vertical arrows are $\LL$-coalgebras (respectively, $\RR$-algebras);
while squares are commutative squares---maps in $\M^\2$---that lift to
maps of $\LL$-coalgebras (respectively, $\RR$-algebras).

This is relevant due to a powerful and slightly surprising result: an
algebraic weak factorization system $(\LL, \RR)$ is completely
determined by either of the double categories
$\UU_\LL \colon \CCoalg \to \Sq{\M}$ or
$\UU_\RR \colon \AAlg \to \Sq{\M}$ over $\Sq\M$;
see~\cite[Theorem~2.24]{riehl-algebraic-model-structures}. This result
was strengthened in~\cite{bourke-garner-awfsI} to give a complete
characterization of when a double category over $\Sq{\M}$ is
isomorphic to the double category of left or right maps for an
algebraic weak factorization system.

\begin{Thm}[{\cite[Theorem~6]{bourke-garner-awfsI}}]\label{thm:bourke}
  A double category $\UU \colon \AA \to \Sq\M$ over $\Sq\M$ is
  isomorphic to the double category of left (resp., right) maps
  for an algebraic weak factorization system on $\M$ if and only if:
\begin{enumerate}[(i),itemsep=0em]
\item The object-level functor $U \colon \A_0 \rightarrow \M$ is an
  isomorphism, and the arrow-level functor $U \colon \A_1
  \rightarrow \M^\2$ is strictly comonadic (resp., monadic); and
\item for every $f \in \A_1$, the square left below (resp., right below) in $\Sq\M$ is in
  the image of $\UU$:
\[ \xymatrix{ a \ar[r]^1 \ar[d]_1 & a \ar[d]^{Uf} \\ a \ar[r]_{Uf} &
    b} \qquad \qquad
\xymatrix{ a \ar[r]^{Uf} \ar[d]_{Uf} & b \ar[d]^1 \\ b \ar[r]_1 & b\rlap{ .}}\]
\end{enumerate}\end{Thm}

This result allows for a straightforward definition and a
straightforward construction of left- and right-liftings for 
algebraic weak factorization systems.

\begin{Defn}
  \label{def:4}
  Given an algebraic weak factorization system $(\LL, \RR)$ on $\M$,
  its \emph{left-lifting} along a left adjoint $V$ or its
  \emph{right-lifting} along a right adjoint $U$ as in~\eqref{eq:10}
  are, when they exist, the algebraic weak factorization systems
  $\smash{(\ll \LL, \ll \RR)}$ on $\K$ and
  $\smash{(\rl \LL, \rl \RR)}$ on $\C$ characterised by the following
  pullbacks of double categories:
\begin{equation}\label{eq:alg-lifting}
  \cd[@C+0.2em]{
    {\CCoalg[\ll \LL]} \ar[r]^-{} \ar[d]_{\UU_{\ll \LL}}
    \save*!/dr-1.7pc/dr:(-1,1)@^{|-}\restore &
    {\CCoalg} \ar[d]^{\UU_\LL} \\
    {\Sq\K} \ar[r]_-{\Sq V} &
    {\Sq\M}
  } \qquad \text{and} \qquad
    \cd[@C+0.2em]{
    {\AAlg[\rl \RR]} \ar[r]^-{} \ar[d]_{\UU_{\rl \RR}}
    \save*!/dr-1.5pc/dr:(-1,1)@^{|-}\restore &
    {\AAlg} \ar[d]^{\UU_\RR} \\
    {\Sq\C} \ar[r]_-{\Sq U} &
    {\Sq\M}\rlap{ .}
  }
\end{equation}
\end{Defn}

\begin{Prop}[{\cite[Proposition~13]{bourke-garner-awfsI}}]
  \label{prop:5}
  Let $(\LL, \RR)$ be an accessible algebraic weak factorization system on $\M$,
  and suppose there are given adjunctions~\eqref{eq:10}.
  If $\C$ and $\K$ are also locally presentable, then $(\LL, \RR)$
  admits both an accessible left-lifting $(\ll \LL, \ll \RR)$ along $V$ and
  an accessible right-lifting $(\rl \LL, \rl \RR)$ along $U$, in the
  sense of Definition~\ref{def:4}.
\end{Prop}

The proof is an application of Theorem~\ref{thm:bourke}: in the
left-lifted case, say, we first pull back $\UU_\LL$ along $\Sq V$ to
obtain a double functor $\UU \colon \mathbb{A} \rightarrow \Sq \K$,
and obtain the desired $(\ll\LL, \ll\RR)$ from this by showing that
$\UU$ satisfies the hypotheses of Theorem~\ref{thm:bourke}. The
only hypothesis which is non-trivial to verify is that
$U_1 \colon \A_1 \rightarrow \K^\2$ is a left adjoint, and for this, we
exploit local presentability and argue exactly as in the proof of
Lemma~\ref{lem:4}.

\subsection{The flaw in the previous proof}
\label{sec:flaw-previous-proof}
We are now in a position to explain the error made in~\cite{HKRS} in
proving Theorem~\ref{thm:2}. The authors state
Proposition~\ref{prop:5} above as Theorems~3.3.1 (for the left case)
and Theorem~3.3.2 (for the right), but add clauses which amount to the
following:
\begin{Cl}\label{claim} In the situation of Definition~\ref{def:4}, if the stated
  left- and right-liftings $\smash{(\ll \LL, \ll \RR)}$ and $\smash{(\rl \LL, \rl
  \RR)}$ of $(\LL, \RR)$ exist, then:
  \begin{enumerate}[(i)]
  \item The underlying weak factorization system of
    $(\smash{\ll \LL, \ll \RR})$ is the left-lifting of the underlying weak
    factorization system of $(\LL, \RR)$; and
  \item The underlying weak factorization system of
    $(\smash{\rl \LL, \rl \RR})$ is the right-lifting of the underlying weak
    factorization system of $(\LL, \RR)$.
  \end{enumerate}
\end{Cl}

This claim would legitimize the following means of constructing the left-
or right-liftings of an accessible weak factorization system as in
Theorem~\ref{thm:2}. One first chooses an accessible algebraic
realization; then lifts that; and then takes the underlying weak
factorization system. The problem with this is that:
\begin{Prop}
  \label{prop:3}
  Claim~\ref{claim} is false.
\end{Prop}
As left-lifting along a left adjoint is the same as right-lifting
along its opposite, it suffices to disprove either (i) or (ii). So let
us concentrate on (i), the case of left-lifting $(\LL, \RR)$ along a
left adjoint $V \colon \K \rightarrow \M$. On the one hand, the
underlying weak factorization system of $(\LL, \RR)$ has left class
$\mathrm{Retr}(\exists \LL)$, and so the left-lifting of this
underlying weak factorization system along $V$ has left class
$V^{-1}(\mathrm{Retr}(\exists \LL))$. On the other hand, the
left-lifting of $(\LL, \RR)$ along $V$ is characterized by a pullback
of double categories as in~\eqref{eq:alg-lifting}; so in particular,
we have a pullback of categories as follows; compare with the
situation of~\eqref{eq:3}:
\begin{equation*}
    \cd[@C+0.2em]{
    {\mathsf{Coalg}_{\ll \LL}} \ar[r]^-{} \ar[d]_{\mathrm{U}_{\ll \LL}}
    \save*!/dr-1.7pc/dr:(-1,1)@^{|-}\restore &
    \coalg \ar[d]^{\mathrm{U}_{\LL}} \\
        {\K^\2} \ar[r]_-{V^\2} &
    {\M^\2}\rlap{ .}
  }
\end{equation*}
Inspecting the images of the vertical functors, we see that a map $f$
of $\K$ admits $\smash{\ll \LL}$-map structure if and only if $Vf$
admits $\LL$-map structure. So
$\exists(\smash{\ll \LL}) = V^{-1}(\exists \LL)$ and the underlying
weak factorization system of $(\smash{\ll \LL, \ll \RR})$ has left
class $\mathrm{Retr}(V^{-1}(\exists \LL))$. This analysis shows that
Claim~\ref{claim}(i) is equally the claim that
\begin{equation}\label{eq:12}
  \mathrm{Retr}(V^{-1}(\exists \LL)) =
  V^{-1}(\mathrm{Retr}(\exists \LL))\rlap{ .}
\end{equation}

Since functors preserve retracts, it will always be the case that
$\mathrm{Retr}(V^{-1}(\exists \LL)) \subseteq
V^{-1}(\mathrm{Retr}(\exists \LL))$; however, the two examples that we
give below below show that, in certain cases, this inclusion is
\emph{strict}. Both of these examples exploit the following general
construction of an algebraic weak factorization system which
originates in~\cite[\sec 4.1]{Bourke2014AWFS2}; we refer the reader to
there for more details.

\begin{Ex}
  \label{ex:2}
  Let $\M$ be a category with finite coproducts, and let $\mathbb{P}$
  be a comonad on $\M$ with counit $\upsilon \colon P \Rightarrow 1$
  and comultiplication $\Delta \colon P \Rightarrow PP$. There is an
  algebraic weak factorization system on $\M$ with functorial
  factorization:
  \begin{equation*}
    X \xrightarrow{f} Y \qquad \mapsto \qquad X \xrightarrow{\iota_1} X
    + PY \xrightarrow{\spn{f, \upsilon_Y}} Y
  \end{equation*}
  and with the fillers $\delta_f$ and $\mu_f$ of~\eqref{eq:4} given by
  the respective composites
  \begin{equation*}
    X + PY \xrightarrow{1_X + P\iota_2\Delta_Y} X + P(X + PY) \
    \text{and} \ X + PY + PY \xrightarrow{1_X + \nabla_{PY}} X + PY\rlap{ .}
  \end{equation*}
  The $\RR$-maps of this algebraic weak factorization system are the
  \emph{$\mathbb{P}$-split epis} $(p,i) \colon X \rightarrow Y$,
  comprising a map $p \colon X \rightarrow Y$ together with a
  ``$\mathbb{P}$-section'': a map $i \colon PY \rightarrow X$ such that
  $pi = \upsilon_Y \colon PY \rightarrow Y$. The $\mathbb{L}$-maps do
  not in general admit a direct description, but the ``algebraically
  cofibrant objects''---the $\mathbb{L}$-maps with domain $0$---are
  precisely the coalgebras for the comonad $\mathbb{P}$.
\end{Ex}

We now give the first of our examples disproving the
equality~\eqref{eq:12}.
\begin{Ex}
  \label{counterexample:1}
  If $A$ is any commutative ring, then there is a weak factorization
  system $(\L, \R)$ on $\Mod_A$ cofibrantly generated by the single
  map $0 \rightarrow A$. The class $\L$ comprises the monomorphisms
  with projective cokernel---so in particular, the $\L$-cofibrant
  objects are the projective modules---while $\R$ comprises the
  epimorphisms; see Lemma~2.2.6 and Proposition~2.2.9 of
  \cite{Hovey1999Model}.

  We obtain an algebraic realization $(\LL, \RR)$ for $(\L, \R)$ using
  Example~\ref{ex:2}, where we take the comonad $\mathbb{P}$ therein
  to be the one generated by the forgetful-free adjunction
  $U \colon \Mod_A \leftrightarrows \cat{Set} \colon F$. In this case,
  the $\mathbb{R}$-maps are $A$-module morphisms
  $f \colon M \rightarrow N$ endowed with a section at the level of
  underlying sets; while an algebraically cofibrant object---a
  $\mathbb{P}$-coalgebra---is easily seen to be a free $A$-module
  endowed with a choice of generators.

  We now specialize to the case $A = \ZZ/6$, so that $\Mod_{A}$ is the
  category of abelian groups in which every element is $6$-torsion. We
  will disprove the equality
  $\mathrm{Retr}(V^{-1}(\exists \LL)) = V^{-1}(\mathrm{Retr}(\exists
  \LL))$ in~\eqref{eq:12} when $V$ is taken to be the left adjoint:
  \begin{equation*}
    V \colon \Mod_{\ZZ/6} \xrightarrow{\ZZ/2 \otimes_{\ZZ/6} (\thg)} \Mod_{\ZZ/6}\rlap{ .}
  \end{equation*}
  
  On the one hand, $0 \rightarrow M$ lies in $\exists \LL$ just when
  $M$ is a free $\ZZ/6$-module. Since the objects in the image of $V$
  are all $2$-torsion, and the only $\ZZ/6$-module which is
  $2$-torsion and free is $0$, it follows that $0 \rightarrow M$ lies
  in $V^{-1}(\exists \LL)$ just when $M$ contains no $2$-torsion
  elements. Since such $M$ can be identified with the $\ZZ/3$-modules,
  they are retract-closed and so, finally, $0 \rightarrow M$ lies in
  $\mathrm{Retr}(V^{-1}(\exists \LL))$ just when $M$ contains no
  $2$-torsion elements.

  On the other hand, a map $0 \rightarrow M$ is in
  $\mathrm{Retr}(\exists \LL)$ just when $M$ is projective; it now
  follows that $0 \rightarrow \ZZ/6$ lies in
  $V^{-1}(\mathrm{Retr}(\exists \LL))$, since $V(\ZZ/6) = \ZZ/2$ is
  projective as a direct summand $\ZZ/6 \cong \ZZ/2 \oplus \ZZ/3$. We
  have thus shown that $0 \rightarrow \ZZ/6$ is in
  $V^{-1}(\mathrm{Retr}(\exists \LL))$ but not in 
  $\mathrm{Retr}(V^{-1}(\exists \LL))$, as desired.
\end{Ex}

The second example is built on the same principle. 

\begin{Ex}
  \label{counterexample:2}
  Let $M$ be the monoid $\{1, e\}$ with $e^2 = e$ and consider
  the category $M\text-\Set$ of $M$-sets endowed with the weak
  factorization system $(\L, \R)$ cofibrantly generated by the single
  map $\emptyset \rightarrow M$. The $\R$-maps are the epimorphisms,
  and as each $M$-set is a retract of a coproduct of $M$'s, each
  object is cofibrant.

  We obtain an algebraic realization $(\LL, \RR)$ for $(\L, \R)$ using
  Example~\ref{ex:2}, where we take the comonad $\mathbb{P}$ therein
  to be the one generated by the free-forgetful adjunction
  $U \colon M\text-\Set \leftrightarrows \cat{Set} \colon F$. Now
  $\mathbb{R}$-maps are maps of $M$-sets endowed with a section of
  underlying sets; while an algebraically cofibrant object is one with
  free $M$-action (the coalgebra structure is, in this case, uniquely
  determined).

  In this situation, we will show
  $\mathrm{Retr}(V^{-1}(\exists \LL)) \subsetneq
  V^{-1}(\mathrm{Retr}(\exists \LL))$ with $V$ taken to be the left
  adjoint functor $\cat{Set} \rightarrow M\text-\cat{Set}$ which
  endows each set with its trivial $M$-action. On the one hand,
  $\emptyset \rightarrow X$ lies in $\exists \LL$ just when $X$ is a
  free $M$-set. Since the trivial action is only free on the empty
  set, we see that $\emptyset \rightarrow X$ lies in
  $V^{-1}(\exists \LL)$, or equally in
  $\mathrm{Retr}(V^{-1}(\exists \LL))$, only when $X = \emptyset$. On
  the other hand, every map of the form $\emptyset \rightarrow X$ is
  in $\mathrm{Retr}(\exists \LL)$, and so every map
  $\emptyset \rightarrow X$ lies in
  $V^{-1}(\mathrm{Retr}(\exists \LL))$. So
  $\mathrm{Retr}(V^{-1}(\exists \LL)) \subsetneq
  V^{-1}(\mathrm{Retr}(\exists \LL))$ as desired.
\end{Ex}

These examples are concerned with lifting factorizations for a single
weak factorization system. If desired, they can be enhanced to
examples concerning lifting factorizations for an accessible model
category by taking $\Cof = \L$ and $\Fib = \R$ and $\WE = $ all maps.
Of course, the model categories so arising are homotopically rather
uninteresting, but in particular cases we may be able to do better.
For instance, a dg version of Example \ref{counterexample:1} occurs in
lifting the (cofibration, acyclic fibration) weak factorization system
of the standard model structure on $\cat{Ch}(\Mod_{\ZZ/6})$.

%

\section{Fixing the previous proof}\label{sec:fixing-prev-proof}
In this final section, we describe how the erroneous Claim~\ref{claim}
can be corrected by adding extra hypotheses, and then show that this
revised claim allows for a correct proof of the algebraic version of Theorem~\ref{thm:2}.
Towards our first goal, let us define an algebraic weak factorization
system $(\LL, \RR)$ to be \emph{left-retract-closed} (resp.,
\emph{right-retract-closed}) if the class of maps $\exists \LL$
(resp., $\exists \RR$) is closed under retracts.
\begin{Prop}
  \label{prop:8}
  Claim~\ref{claim}(i) holds for any $(\LL, \RR)$ which is
  left-retract-closed, while Claim~\ref{claim}(ii) holds for any
  right-retract-closed $(\LL, \RR)$.
\end{Prop}
\begin{proof}
  The two cases are dual, so it suffices to consider a
  left-retract-closed $(\LL, \RR)$ on $\M$ and a left adjoint
  $V \colon \K \rightarrow \M$ along which the left-lifting
  $(\smash{ \ll \LL, \ll \RR})$ exists. We must prove the
  equality~\eqref{eq:12}. We already noted that
  \begin{equation*}
    \mathrm{Retr}(V^{-1}(\exists \LL)) \subseteq
    V^{-1}(\mathrm{Retr}(\exists \LL))
  \end{equation*}
  since functors preserve retracts. Conversely, because
  $\exists \LL = \mathrm{Retr}(\exists \LL)$, we have 
  \begin{equation*}
    V^{-1}(\mathrm{Retr}(\exists \LL)) = V^{-1}(\exists \LL) \subseteq
    \mathrm{Retr}(V^{-1}(\exists \LL))\rlap{ .} \qedhere
  \end{equation*}
\end{proof}

This suggests the following legitimate construction of the left- or
right-liftings of an accessible weak factorization system $(\L, \R)$.
One first chooses a \emph{left-retract-closed} (resp.,
right-retract-closed) accessible algebraic realization; then lifts
that; and then takes the underlying weak factorization system.

In order for this to work, the required left- and right-retract-closed
algebraic realizations of $(\L, \R)$ must exist. Since we already know
that at least one accessible algebraic realization $(\LL, \RR)$
exists, it suffices to show that this can be adjusted to a
left-retract-closed one $(\LL^\sharp, \RR^\sharp)$ and a
right-retract-closed one $(\LL^\flat, \RR^\flat)$ with the same
underlying weak factorization system.

The idea is to construct the adjustment $(\LL^\sharp, \RR^\sharp)$ in
such a way that the $\LL^\sharp$-maps are precisely the cloven
$\L$-maps: for then, by Lemma~\ref{lem:4}, $\exists(\LL^\sharp) = \L$,
which is indeed closed under retracts; moreover, the underlying weak
factorization system is clearly the same. Dually, we will construct
$(\LL^\flat, \RR^\flat)$ such that the $\RR^\flat$-maps are the cloven
$\R$-maps. In fact, by Theorem~\ref{thm:2}, these motivating
\emph{descriptions} of $(\LL^\sharp, \RR^\sharp)$ and
$(\LL^\flat, \RR^\flat)$ are nearly sufficient for their
\emph{construction}. The only additional aspect that is required is:

\begin{Prop}
  \label{prop:7}
  Let $(\LL, \RR)$ be an accessible algebraic weak factorization
  system. The cloven $\L$-maps admit a composition law
  $\cat{Clov}(\L) \times_\M \cat{Clov}(\L) \rightarrow \cat{Clov}(\L)$
  making them the vertical morphisms and squares of a double
  category $\mathbb{C}\cat{lov}(\L) \rightarrow \Sq \M$ over $\Sq \M$
  whose objects and horizontal morphisms are those of $\M$. Dually,
  the cloven $\R$-maps constitute a double category
  $\mathbb{C}\cat{lov}(\R) \rightarrow \Sq \M$.
\end{Prop}
\begin{proof}
  By duality, we need only consider the left case. Our proof
  follows~\cite[\sec 2.7]{bourke-garner-awfsI}. To begin with, we
  define an \emph{algebra lifting operation} for a map
  $f \colon A \rightarrow B$ to be the choice, for each $\RR$-algebra
  $(g,p) \colon C \rightarrow D$ and each map
  $(h,k) \colon f \rightarrow g$ in $\M^\2$ of a diagonal filler
  $\varphi_{(g,p)}(h,k) \colon B \rightarrow C$:
  \begin{equation}\label{eq:15}
    \cd[@+0.5em]{
      {A} \ar[r]^-{h} \ar[d]_{f} &
      {C} \ar[d]^{g} \\
      {B} \ar[r]_-{k} \ar@{-->}[ur]|-{\varphi(h,k)} &
      {D} \rlap{ ,}
    }
  \end{equation}
  subject to the naturality condition that, for any map
  $(u,v) \colon (g,p) \rightarrow (h,r)$ of $\mathbb{R}$-algebras, we
  have $u \varphi_{(g,p)}(h,k) = \varphi_{(h,r)}(uh,vk)$.

  Now, a square like~\eqref{eq:15} is equally an object of the comma
  category $f \downarrow \mathrm{U}_\RR$, and the unit map
  $(Lf, 1) \colon f \rightarrow \mathrm{U}_\RR(Rf, \mu_f)$ is initial
  in this comma category; so to give $\varphi$ is equally to give a
  single map $\varphi_{(Rf, \mu_f)}(Lf, 1) \colon B \rightarrow Ef$
  filling the left square of~\eqref{eq:14}. In this way, we obtain an
  isomorphism $\cat{Clov}(\L) \cong \cat{Lift}_\RR$ over
  $\smash{\M^\2}$, where $\cat{Lift}_\RR$ is the category of maps
  endowed with algebra lifting operations and squares commuting with
  the lifting operations.

  We may now exploit this isomorphism to define the desired
  composition law on algebra lifting operations rather than on cloven
  $\L$-maps. Given maps $f \colon A \rightarrow B$ and
  $g \colon B \rightarrow C$ endowed with lifting operations $\varphi$
  and $\psi$, we obtain a composite lifting operation $\psi \varphi$
  on $gf$ by first lifting against $f$ and then against $g$:
  \begin{equation*}
    \psi \varphi_{(h,p)}(u,v) = \psi_{(h,p)}(\varphi_{(h,p)}(u,vg), v) \qquad \quad
  \vcenter{
  \xymatrix@C=8pc{ A \ar[d]_{f} \ar[r]^u & D \ar[dd]^h \\ B
    \ar@{-->}[ur]^(.4){\phi(u,vg)} \ar[d]_g & \\ C \ar[r]_v
    \ar@{-->}[uur]_{\psi(\phi(u,vg),v)} & E\rlap{ .}
  }}
  \end{equation*}
  This assignation is easily functorial with respect to maps of
  lifting operations, thus yielding a functor
  $\cat{Lift}_\RR \times_\M \cat{Lift}_\RR \rightarrow
  \cat{Lift}_\RR$. To see that this gives rise to the desired double
  category, we must check associativity and unitality of this
  composition law. Associativity is immediate on comparing the
  formulae for $\xi(\psi \varphi)$ and $(\xi \psi)\varphi$; while an
  identity at $A$ is easily seen to be given by the lifting structure
  $(1_A, \iota_A) \colon A \rightarrow A$ with
  $(\iota_A)_{(g,p)}(u,v) = u$.
\end{proof}
\begin{Rk}
  \label{rk:2}
  The double categories $\mathbb{C}\cat{lov}(\L)$ and
  $\mathbb{C}\cat{lov}(\R)$ are in fact expansions of the double
  categories $\CCoalg$ and $\AAlg$ of $\LL$- and $\RR$-maps: the above
  proof simply repeats the construction of the composition laws on the
  latter in the broader context. In particular, this means that there
  are canonical inclusion double functors
  $\CCoalg \hookrightarrow \mathbb{C}\cat{lov}(\L)$ and
  $\AAlg \hookrightarrow \mathbb{C}\cat{lov}(\R)$ over $\Sq \M$.
\end{Rk}
We now use these double categories to build the desired left- and
right-shifted algebraic weak factorization systems.

\begin{Prop}\label{prop:shifted-awfs}
  Let $(\LL,\RR)$ be an accessible algebraic weak factorization system
  on a locally presentable category $\M$. There exist accessible
  algebraic weak factorization systems $(\LL^\sharp,\RR^\sharp)$ and
  $(\LL^\flat, \RR^\flat)$ characterised by isomorphisms of double
  categories
  \begin{equation}\label{eq:16}
    \CCoalg[\LL^\sharp] \cong \mathbb{C}\cat{lov}(\L) \qquad
    \text{and} \qquad
    \AAlg[\RR^\flat] \cong \mathbb{C}\cat{lov}(\R)
  \end{equation}
  over $\Sq \M$. Furthermore, $(\LL^\sharp,\RR^\sharp)$ is
  left-retract-closed, $(\LL^\flat, \RR^\flat)$ is
  right-retract-closed, and both have the same underlying weak
  factorization system as $(\LL, \RR)$.
\end{Prop}

\begin{proof}
  We verified the final sentence above; as for the existence of
  $(\LL^\sharp, \RR^\sharp)$ and $(\LL^\flat, \RR^\flat)$, the 
  arguments involve applying Theorem \ref{thm:bourke} to the double
  categories $\mathbb{C}\cat{lov}(\L)$ and $\mathbb{C}\cat{lov}(\R)$
  over $\Sq \M$. We give details only in the left case.

  For hypothesis (i) of Theorem~\ref{thm:2}, the object-level functor
  $1_\M \colon \M \rightarrow \M$ is clearly an isomorphism, while the
  arrow-level functor
  $\mathrm{U}_\L \colon \cat{Clov}(\L) \rightarrow \M^\2$ has a right
  adjoint by Lemma~\ref{lem:4}, and is therefore comonadic because it
  is the forgetful functor from the category of coalgebras for a
  copointed endofunctor; see~\cite[\sec 5.1]{Kelly1980A-unified}, for
  example.

  For hypothesis (ii), note first that the \emph{unique} cloven
  $\L$-map structure on an identity map $1_A \colon A \rightarrow A$
  is given by $(1_A, L1_A) \colon A \rightarrow A$. To verify (ii)
  therefore, we must show that any cloven $\L$-map
  $(f,s) \colon A \rightarrow B$, the square left below lifts to a map
  $(1_A, L1_A) \rightarrow (f,s)$ of cloven $\L$-maps.
  \begin{equation*}
    \cd[@!C@C-0.3em]{
      {A} \ar[r]^-{1_A} \ar[d]_{1_A} &
      {A} \ar[d]^{f} & &
      \sh{l}{0.3em}{E1_A} \ar[r]^-{E(1_A, f)} &
      \sh{r}{0.3em}{Ef} \\
      {A} \ar[r]_-{f} &
      {B} & &
      {A} \ar[r]_-{f} \ar[u]^{L1_A} &
      {B} \ar[u]_{s} 
    }
  \end{equation*}
  This is equally to show the commutativity of the square above right;
  for which we calculate that $E(1_A,f) \circ L1_A = Lf = sf$.
\end{proof}

\begin{Rk}
  The inclusion double functors of Remark~\ref{rk:2} compose with the
  isomorphisms~\eqref{eq:16} to yield double functors
  $\CCoalg \rightarrow \CCoalg[\LL^\sharp]$ and
  $\AAlg \rightarrow \AAlg[\RR^\sharp]$ over $\Sq \M$. The existence
  of these double functors can be equivalently expressed as saying
  that we have \emph{oplax} (= ``left Quillen'') morphisms of
  algebraic weak factorization systems
  $(\LL^\flat, \RR^\flat) \rightarrow (\LL, \RR) \rightarrow
  (\LL^\sharp, \RR^\sharp)$ with underlying functor the identity;
  see~\cite[Lemma~6.9]{riehl-algebraic-model-structures}.
\end{Rk}

Using the preceding proposition, we can finally give:

\begin{proof}[Proof of Theorem~\ref{thm:2} (bis)]
  Given the accessible weak factorization system $(\L, \R)$ on $\M$,
  we first choose an accessible algebraic realization $(\LL, \RR)$. In
  the left-lifted case, we then replace this with the
  left-retract-closed realization $(\LL^\sharp, \RR^\sharp)$ given by
  Proposition~\ref{prop:shifted-awfs}. Now by
  Proposition~\ref{prop:5}, this admits a left-lifting along
  $V \colon \K \rightarrow \M$ to an accessible algebraic weak
  factorization system $(\ll{\LL}^\sharp, \ll{\RR}^\sharp)$ on $\K$.
  Since we are in the left-retract-closed situation,
  Proposition~\ref{prop:8} ensures that the underlying weak
  factorization system of $(\ll{\LL}^\sharp, \ll{\RR}^\sharp)$ is the
  desired accessible left-lifting of $(\L, \R)$ along $V$. The case of
  right-lifting is entirely dual.
\end{proof}

\begin{Rk}
  \label{rk:1}
  In giving the preceding proof, we treated the left- and right-lifted
  cases entirely symmetrically; however, in practice there is an
  asymmetry. The proof of Proposition~\ref{prop:1} above, which we
  omitted, involves the construction of a particular accessible
  algebraic realization $(\LL, \RR)$ for each given accessible
  $(\L, \R)$. It turns out that this particular $(\LL, \RR)$ is always
  right-retract-closed, since its category of $\RR$-maps is
  \emph{cofibrantly generated by a small category} in the sense
  of~\cite{Garner2009Understanding}. Thus, so long as this particular
  algebraic realization is chosen, there is no need to make an
  adjustment in the right-lifted case. This point was already spelt
  out by the third author
  in~\cite[Theorem~2.10]{riehl-algebraic-model-structures}. 
\end{Rk}

\bibliographystyle{acm}
\bibliography{./biblio}

\end{document}